\documentclass[10pt]{amsart}
\usepackage[table]{xcolor}
\usepackage{amsmath,amsthm,amssymb,mathtools, scrextend,enumerate,bbm,tikz-cd,verbatim,cases,thmtools}
\usepackage{mathrsfs}
\usepackage{tikz-cd}
\usepackage[utf8]{inputenc}
\PassOptionsToPackage{hyphens}{url}\usepackage{hyperref}\PassOptionsToPackage{hyphens}{url}\usepackage{hyperref}

\usepackage{fancyhdr}
\allowdisplaybreaks

\newtheorem{thm}{Theorem}[section]
\newtheorem{lem}[thm]{Lemma}
\newtheorem{prop}[thm]{Proposition}

\theoremstyle{definition}

\newtheorem{rem}[thm]{Remark}

\newtheorem*{defn*}{Definition}
\newtheorem*{thm*}{Theorem}
\newtheorem*{prop*}{Proposition}
\newtheorem*{cor*}{Corollary}

\newtheorem*{question*}{Question}

\newcommand{\C}{\mathbb{C}}
\newcommand{\R}{\mathbb{R}}
\renewcommand{\Re}{\operatorname{Re}}

\newcommand{\h}{\mathbb{H}}
\renewcommand{\O}{\mathbb{O}}

\newcommand{\Spin}{\operatorname{Spin}}
\newcommand{\Proj}{\operatorname{Proj}}

\newcommand{\Tr}{\operatorname{Tr}}

\newcommand{\Id}{\operatorname{Id}}

\newcommand{\Span}{\operatorname{Span}}

\newcommand{\diag}{\operatorname{diag}}

\renewcommand{\ker}{\operatorname{ker}}

\begin{document}
 

\title{Stable submanifolds in the product of projective spaces II}

\author{Shuli Chen And Alejandra Ramirez-Luna}
\date{\today}
 \maketitle

\begin{abstract} We prove that there do not exist odd-dimensional stable compact minimal immersions in the product of two complex projective spaces. We also prove that the only stable compact minimal immersions in the product of a quaternionic projective space with any other Riemannian manifold are the products of quaternionic projective subspaces with  compact stable minimal immersions of the second manifold in the Riemmanian product. These generalize similar results of the second-named  author of immersions with low dimensions or codimensions to immersions with arbitrary dimensions. In addition, we prove that the only stable compact minimal immersions in the product of a octonionic projective plane with any other Riemannian manifold are the products of octonionic projective subspaces with compact stable minimal immersions of the second manifold in the Riemmanian product.
\end{abstract}
\section{Introduction}
A classical and fascinating problem in Riemannian geometry is the study of submanifolds that minimize area under perturbations in a given Riemannian manifold. This gives rise to the the study of stable minimal submanifolds, which are critical points of the area functional and whose second variation of the area functional is non-negative (equivalently, the mean curvature vector and Morse index both equal to zero \cite{simons1968minimal}).
In particular, it is an interesting problem to obtain geometric information of the submanifold by just knowing that it is stable.\\

A lot of research has been done in that direction where the ambient manifold is well-known. For example, Fischer-Colbrie and Schoen \cite{fischer1980structure}, do Carmo and Peng \cite{do1979stable}, and Pogorelov \cite{pogorelov1981stability} independently proved that planes are the only stable complete minimal surfaces in the $3$-dimensional Euclidean space, and recently Chodosh and Li proved that a complete, two-sided, stable minimal hypersurface in the  $4$-dimensional Euclidean space must be flat \cite{chodosh2021stable}. When the ambient manifold is compact, Simons proved that the are no compact stable minimal submanifolds in the Euclidean sphere \cite{simons1968minimal}. Later, Lawson and Simons characterized the complex submanifolds (in the sense that each tangent space is invariant under the complex structure) as the only compact stable minimal submanifolds in the complex projective space \cite{lawson1973stable}. Finally, Ohnita completed the classification of compact stable minimal submanifolds in all compact rank one symmetric spaces by proving that the only stable submanifolds in the real and quaternionic projective space,  and the Cayley plane, are the real and quaternionic projective subspaces, and the Cayley projective line, respectively \cite{ohnita1986stable}.\\

In the case where the ambient manifold is a Riemannian product of well-known Riemannian manifolds we have the following: Torralbo and Urbano proved a characterization of stable submanifolds in the product of the Euclidean sphere with any other Riemannian manifold \cite{torralbo2014stable}, and Chen and Wang proved a similar characterization of stable submanifolds in the product of any hypersurface of the Euclidean space with certain conditions and any Riemannian manifold \cite{chen2013stable}.\\

Along these lines and following similar ideas of Torralbo, Urbano, Chen, and Wang, Ramirez-Luna in \cite{Ramirez3} proved a characterization theorem for stable submanifolds of specific dimensions in the product of a complex and quaternionic projective space with any other Riemannian manifold. In particular we recall the following results: 
\begin{thm}\cite{Ramirez3}
\label{M3}
The only compact stable minimal immersions of codimension $d=2$ or dimension $n=2$ in the product manifold $\bar{M}:=\mathbb{C}P^{\frac{m_{1}}{2}}\times \mathbb{C}P^{\frac{m_{2}}{2}}$ are the complex ones, in the sense that each tangent space is invariant under the complex structure $J_{1}$ or $J_{2}$ of $\bar{M}$ (see Definition 3.7 in \cite{Ramirez3}). 
\end{thm}

\begin{thm} \cite{Ramirez3}
\label{QUATERNIONICTHEOREM}
Let $\Phi=(\psi,\phi):\Sigma \rightarrow \mathbb{H}P^{\frac{m_{1}}{4}}\times M$ be a compact stable minimal immersion of codimension $d$ and dimension $n$, where $M$ is any Riemannian manifold of dimension $m_{2}$. Then, 

\begin{itemize}
    \item If $d=1$,
     $\Sigma=\mathbb{H}P^{\frac{m_{1}}{4}}\times \hat{\Sigma}$, $\Phi =  \Id \times \hat{\phi}$ where $\hat{\phi}:\hat{\Sigma}\rightarrow M$ is a compact stable minimal  immersion of codimension $1$,  and therefore $\Phi(\Sigma)=\mathbb{H}P^{\frac{m_{1}}{4}}\times \hat{\phi}(\hat{\Sigma})$.
     In particular, for $m_{2}=1$, $\Sigma=\mathbb{H}P^{\frac{m_{1}}{4}}$, $\hat{\phi}$ is a constant function, and $\Phi(\Sigma)=\mathbb{H}P^{\frac{m_{1}}{4}}\times \{q\}$, for $q\in M$.
\item If $d=2$,
     $\Sigma=\mathbb{H}P^{\frac{m_{1}}{4}}\times \hat{\Sigma}$, $\Phi= \Id\times \hat{\phi}$ where $\hat{\phi}:\hat{\Sigma}\rightarrow M$ is a compact stable minimal  immersion of codimension $2$,  and therefore $\Phi(\Sigma)=\mathbb{H}P^{\frac{m_{1}}{4}}\times \hat{\phi}(\hat{\Sigma})$.
     In particular, for $m_{2}=1$, there are no compact stable minimal immersions of codimension $2$ in $\mathbb{H}P^{\frac{m_{1}}{4}}\times M$. And for $m_{2}=2$, $\Sigma=\mathbb{H}P^{\frac{m_{1}}{4}}$, $\hat{\phi}$ is a constant function, and $\Phi(\Sigma)=\mathbb{H}P^{\frac{m_{1}}{4}}\times \{q\}$, for $q\in M$.
    \item If $n=1$, $\phi:\Sigma \rightarrow M$ is a  stable geodesic, $\psi$ is a constant function,  and therefore $\Phi(\Sigma)=\{r\}\times \phi(\Sigma)$ with $r$ a point of $\mathbb{H}P^{\frac{m_{1}}{4}}$.
    
    \item If $n=2$, $\phi:\Sigma \rightarrow M$ is a stable minimal immersion of dimension $2$, $\psi$ is a constant function, and therefore $\Phi(\Sigma)=\{r\}\times \phi(\Sigma)$ with $r$ a point of $\mathbb{H}P^{\frac{m_{1}}{4}}$.
\end{itemize}
\end{thm}

Recall that the technique in these problems involving stable submanifolds, which goes back to Simons \cite{simons1968minimal}, is to find appropriate normal sections, such that when we add the second variations along them, we obtain a non-positive sign. This together with the stability of the submanifold, allows us in some cases to obtain geometric information about the submanifold. Therefore, an essential part of the proofs of Theorems  \ref{M3} and \ref{QUATERNIONICTHEOREM} is to prove that Equations (9) and (10) in Lemma 3.1 and (47) and (48) in Lemma 4.1 in \cite{Ramirez3} for the specific dimensions and codimensions have a non-positive sign (we call this the vector inequality). In \cite{Ramirez3}, the vector inequality is proved in dimension one and two.
In this paper we prove  that the same vector inequality holds in general dimensions (see Proposition \ref{ineq: key result}), and prove another inequality (see Proposition \ref{prop: octo ineq}). This allows us to prove the nonexistence of odd-dimensional stable minimal submanifolds in the product of two complex projective spaces, and obtain a complete characterization of stable submanifolds in the product of a quaternionic projective space with any other Riemannian manifold, and stable minimal submanifolds  in the product of a octonionic projective plane with any other Riemannian manifold.

From Theorem \ref{M3} (see also \cite{torralbo2014stable} and references in there), it is expected that stable minimal submanifolds in a Riemannian manifold with a complex structure behave well under the same complex structure. Towards this idea we prove that:

\begin{restatable}{thm}{mainone}
\label{main1}
There do not exist odd-dimensional stable compact minimal immersions in $\C P^{m_1/2} \times \C P^{m_2/2}$.
\end{restatable}
For even-dimensional stable compact minimal immersions in $\C P^{m_1/2} \times \C P^{m_2/2}$, even though we obtain a non-positive sign in the second variation, this is not enough to conclude that the tangent space of the stable submanifold behaves well under a complex structure of $\C P^{m_1/2} \times \C P^{m_2/2}$. \\

Theorem \ref{QUATERNIONICTHEOREM} shows that the stable minimal submanifolds of dimension or codimension $1$ or $2$ in $\mathbb{H}P^{\frac{m_{1}}{4}}\times M$ are precisely the product of stable minimal submanifolds in $\mathbb{H}P^{\frac{m_{1}}{4}}$ and $M$. In this paper, we generalize this result to arbitrary dimension:

 \begin{restatable}{thm}{maintwo}
\label{main2}
 Let $\Phi:\Sigma \to \h P^{m_1/4} \times M$ be a stable compact minimal immersion, where $M$ is a Riemannian manifold of dimension $m_{2}$. Then $\Sigma = \h P^{m_1'/4} \times \Sigma_2$, $\Phi = \psi_1 \times \psi_2$, where $\psi_1: \h P^{m_1'/4} \to \h P^{m_1/4}$, $\psi_2: \Sigma_2 \to M$ are stable compact minimal immersions.
\end{restatable}

In this paper, we further consider the stable compact minimal immersions of arbitrary dimension in $\mathbb{O}P^{2}\times M$. The linear algebra in this case is more involved due to non-associativity of octonions (see Proposition \ref{prop: octo ineq}). Analogous to Theorem \ref{main2}, we obtain the following result:
 \begin{restatable}{thm}{mainthree}
\label{main3}
 Let $\Phi:\Sigma \to \O P^{2} \times M$ be a stable compact minimal immersion, where $M$ is a Riemannian manifold of dimension $m_{2}$. Then $\Sigma = \Sigma_1 \times \Sigma_2$, $\Phi = \psi_1 \times \psi_2$, where $\Sigma_1$ is either a point, $\O P^{1}$ or $\O P^{2}$, and $\psi_1: \Sigma_1 \to \O P^{2}$, $\psi_2: \Sigma_2 \to M$ are stable compact minimal immersions.
\end{restatable}

This paper is structured as follows: In Section \ref{vectorinequality} we prove the vector inequality and necessary linear algebra facts for the following section. In Section \ref{geometricapplications}, we give some preliminaries, and prove Theorems \ref{main1}, \ref{main2}, and \ref{main3}.\\

 \textbf{Acknowledgments.}
  The authors are grateful for the patience and useful discussions and suggestions of Otis Chodosh. S. C. is also grateful for useful discussions with Lie Qian. S. C. is partially sponsored by the Mary V. Sunseri Graduate Fellowship at Stanford University.
  
\section{The Vector Inequality}
\label{vectorinequality}

First we prove some basic results in linear algebra that will be needed later.

\begin{prop}\label{linearalgebra}
Let $X$ be an $m\times n$ matrix. Then,

\begin{itemize}
    \item the row spaces of $X^TXX^T$ and $X^T$ are the same.
    \item the symmetric matrices $X^TX$ and $XX^T$ have the same nonzero eigenvalues with the same multiplicity. 
\end{itemize}
\end{prop}

\begin{proof}
\begin{itemize}
\item Let us denote by $R(Y)$ the row space of a matrix $Y$. It is clear that $R(X^TXX^T) \subset R(X^T)$. For the reverse direction, it is enough to show $\ker(X^TXX^T) \subset \ker(X^T)$. Let $y \in \ker(X^TXX^T)$. Then we have
$$ 0 =   y^TXX^TXX^T y  = \|XX^Ty\|^2,$$
so $XX^Ty = 0$. Thus
$$ 0 =   y^TXX^Ty  = \|X^Ty\|^2,$$
showing $X^Ty =0$, so $y \in \ker(X^T)$. Thus $\ker(X^TXX^T) \subset \ker(X^T)$ as desired. 

\item Let $\lambda$ be a nonzero eigenvalue of $X^TX$ with multiplicity $k$. Then there exist $k$ orthonormal eigenvectors $v_1, \dots, v_k$ with eigenvalue $\lambda$. Then, $(XX^T)(X v_i) = X(X^TX) v_i = \lambda X v_i$, so $Xv_i$ are eigenvectors of $XX^T$ with eigenvalue $\lambda$. Moreover, we claim they are linearly independent. Let $\alpha_i$ be such that $\sum_i \alpha_i Xv_i = 0$. Then $ 0 = X^T(\sum_i \alpha_i Xv_i) = \sum_i \alpha_i X^TXv_i = \sum_i \alpha_i \lambda v_i = \lambda \sum_i \alpha_i v_i$. Since $v_i$'s are linearly independent, we must have $\alpha_i = 0$ for each $i$. Thus $Xv_i$'s are linearly independent. 

Reverting the roles of $XX^T$ and $X^TX$ shows the $\lambda$-eigenspaces of $XX^T$ and $X^TX$ have the same dimension. Thus $X^TX$ and $XX^T$ have the same nonzero eigenvalues with multiplicity as desired.

\end{itemize}
\end{proof}

For any two $m \times n$ matrices $A, B \in M_{m\times n}(\R)$, one can define the \emph{Frobenius} or \emph{trace inner product} as

$$\langle A, B \rangle := \Tr(A^TB) = \sum_{i=1}^m \sum_{j=1}^n a_{ij}b_{ij},$$
which induces the norm
$$\| A\|^2:= \Tr(A^TA) = \sum_{i=1}^m \sum_{j=1}^n a_{ij}^2.$$

This is the usual Eulidean inner product by identifying $M_{m\times n}(\R)$ with $\R^{mn}$.

We thus have the familiar Cauchy-Schwarz inequality:

$$\langle A, B \rangle \le \|A\|\|B\|,$$
and we obtain equality if and only if $A$ and $B$ are nonnegative multiples of each other. 

The particular nice thing about the trace operator is that we have 
$$\Tr(A^TB) = \Tr(BA^T)$$

for any $A,B \in M_{m\times n}(\R)$.
\begin{prop}
 \label{ineq: key result}
Let $A \in O(m)$ be an orthogonal matrix. Let $x_1, \dots x_n$ be vectors in $\R^m$. Let $X$ be the associated $m \times n$ matrix
$$X = \begin{bmatrix}x_1 & \dots & x_n \end{bmatrix}.$$
Then 
$$\sum_{i,j=1}^{n} \langle  x_i, Ax_j \rangle^2 - \sum_{i,j=1}^{n} \langle  x_i, x_j \rangle^2  \le 0,$$
and we have equality if and only if $XX^T$ commutes with $A$. In particular, when we have equality, $\Span\{x_1, \dots, x_n\}$ is $A$-invariant. 

\end{prop}

\begin{proof}
Notice that $X^TX = [\langle x_i,  x_j\rangle]_{i,j=1}^n$, and $X^TAX = [\langle x_i,  Ax_j\rangle]_{i,j=1}^n$, 
so 

\begin{center}
$ \displaystyle\|X^TX\|^2 =\sum_{i,j=1}^{n} \langle  x_i, x_j \rangle^2$, and 
    $\displaystyle\| X^TAX \|^2 =\sum_{i,j=1}^{n} \langle  x_i, Ax_j \rangle^2$.
\end{center}
Moreover, we also have that 
\begin{center}
$\|XX^T\|^2 = \Tr(XX^TXX^T) = \Tr(X^TXX^TX) =  \|X^TX\|^2$. 
\end{center}
We calculate 
\begin{align*}
 \| A^TXX^TA \|^2
 &= \Tr(A^TXX^TAA^TXX^TA ) \\
 & = \Tr(A^TXX^TXX^TA ) \\
 & = \Tr(AA^TXX^TXX^T) \\
  & = \Tr(XX^TXX^T) \\
  & =  \| XX^T  \|^2.
\end{align*}

Thus
\begin{align*}
\sum_{i,j=1}^{n} \langle  x_i, Ax_j \rangle^2   &= \| X^TAX\|^2 \\
& = \Tr((X^TAX)^TX^TAX) \\
& = \Tr((X^TA^TXX^TA)X) \\
& = \Tr(X(X^TA^TXX^TA)) \\
& = \Tr((XX^T)(A^TXX^TA)) \\
& = \langle XX^T,  A^TXX^TA \rangle  \\
& \le \| XX^T\|\| A^TXX^TA\| \\
& = \| XX^T\|\|XX^T\| \\
& = \| XX^T\|^2 \\
& = \sum_{i,j=1}^{n}\langle  x_i, x_j \rangle^2,
\end{align*}
as desired. We have equality if and only if we obtain equality in the Cauchy-Schwarz inequality, which happens if and only if $XX^T$ and $A^TXX^TA$ are nonnegative multiples of each other. Since $\| A^TXX^TA \| = \|XX^T \|$, we have equality if and only if $A^TXX^TA = XX^T$, that is, when $XX^TA = AXX^T$.

Lastly, suppose we have $A^TXX^TA = XX^T$. Multiplying $X^T$ on the left to this equation yields $X^TXX^T = X^TA^TXX^TA = X^TA^TX(A^TX)^T$. This shows that the row space of $X^T X X^T$ is contained in the the row space of $(A^TX)^T$. By Proposition \ref{linearalgebra}, this means the row space of $X^T$ is contained in the row space of $(A^TX)^T$, so $\Span\{x_1, \dots, x_n\} \subset \Span\{A^T(x_1), \dots, A^T(x_n)\}$. Since they have the same dimension, we have $\Span\{x_1, \dots, x_n\} = \Span\{A^T(x_1), \dots, A^T(x_n)\}$, so $\Span\{x_1, \dots, x_n\}$ is $A^T$-invariant. Since $A$ is orthogonal, $A^{T}= A^{-1}$, $A$-invariant spaces and $A^T$-invariant spaces coincide.
\end{proof}

We also have the following nice result.
\begin{prop}
 \label{ineq: sum}
Let $x_1, \dots x_n$ be vectors such that $\dim \Span\{x_1, \dots, x_n\} = m$. 
Then 
$$m\sum_{i,j=1}^{n} \langle  x_i, x_j \rangle^2  \ge \sum_{i,j=1}^{n}|x_i|^2|x_j|^2,$$
and we have equality if and only if all the $x_i$'s have the same length and are mutually orthogonal.

\end{prop}
\begin{proof}
Without loss of generality we can assume $x_1, \dots x_n \in \R^m$. Let $X$ be the associated $m \times n$ matrix
$$X = \begin{bmatrix}x_1 & \dots & x_n \end{bmatrix}.$$ 

Then $\|XX^T\|^2 = \|X^TX\|^2 = \sum_{i,j=1}^{n} \langle  x_i, x_j \rangle^2 $, and $\Tr(XX^T) = \Tr(X^TX) = \sum_{i=1}^n |x_i|^2$.

Using Cauchy-Schwarz inequality, we have
\begin{align*}
m\sum_{i,j=1}^{n} \langle  x_i, x_j \rangle^2 & = \|\Id_m\|^2   \|XX^T\|^2 \\
& \ge \langle \Id_m, XX^T \rangle^2 \\
& = \Tr(XX^T)^2 \\
& = (\sum_{i=1}^n |x_i|^2)^2 \\
& = \sum_{i,j=1}^{n}|x_i|^2|x_j|^2.
\end{align*}

To have equality, we need $XX^T$ to be a scalar multiple of the identity. By Proposition \ref{linearalgebra}, $X^TX$ is a diagonalizable matrix with all eigenvalues equal. Thus it must also be a scalar multiple of the identity. This happens if and only if all the $x_i$'s have the same length and are mutually orthogonal.
\end{proof}

\section{Main results}
\label{geometricapplications}
\subsection{Preliminaries}
\subsubsection{Second variation formula}
Let $\Phi:(\Sigma,g) \rightarrow (M,\bar{g})$ be a minimal isometric immersion of an $n$-dimensional compact Riemannian manifold $(\Sigma,g)$. Let $C^\infty(\Phi^*TM))$ denote the space of all $C^\infty$-vector fields along $\Phi$. For any $V$ in $C^\infty(\Phi^*TM)$, let $\{\Phi_t\}$ be a $C^\infty$-one-parameter family of immersions of $\Sigma$ into $M$ such that $\Phi_0 = \Phi$ and $\frac{d}{dt}\Phi_t(x)\vert_{t = 0} =
V_x$ for each $x \in \Sigma$. Then we have the classical second variational formula

\begin{thm}[Second Variation formula]
\label{SVF} Under the above assumptions,
\begin{center}
    $\displaystyle \frac{d^{2}}{dt^{2}} |\Phi_t(\Sigma)|\biggr\vert_{t = 0}=-\int_{\Sigma}\langle  J_{\Sigma}V^N,V^N \rangle d\Sigma$,
\end{center}
where $V^N$ the normal
component of $V$ and $J_{\Sigma}$ is the elliptic Jacobi operator on the normal bundle $N(\Sigma)$ defined by\\

\begin{center}
    $\displaystyle J_{\Sigma}(X):=\Delta^{\perp}X+(\sum_{i=1}^{n}R^{M}(X,e_{i})e_{i})^{\perp}+\sum_{i,j=1}^{n} \langle  B(e_{i},e_{j}),X \rangle B(e_{i},e_{j}) $,
\end{center}
and the normal Laplacian on the normal bundle $N(\Sigma)$ is given by

\begin{center}
    $\displaystyle \Delta^{\perp}X=\sum_{i=1}^{n}(\nabla^{\perp}_{e_{i}}\nabla^{\perp}_{e_{i}}X-\nabla^{\perp}_{(\nabla_{e_{i}}e_{i})^{T}}X)$.
\end{center}
Here, $\{e_{1},\ldots,e_{n}\}$ is an orthonormal basis of $T\Sigma$, $\nabla$ is the connection of $M$, $\nabla^{\perp}$ is the normal connection of $\Sigma$ in $M$, $B$ is the second fundamental form of $\Sigma$ in $M$, and $R^{M}$ is the curvature tensor of $M$. 
\end{thm}

We say a minimal immersion $\Phi$ is \emph{stable} if for all $V \in C^\infty(\Phi^*\Gamma(M))$,
$$-\int_{\Sigma}\langle  J_{\Sigma}V^N,V^N \rangle \, d\Sigma \ge 0.$$

\subsubsection{Immersion into a Riemannian product}
\label{subsubsec: imm}

The general setting of our problem is as follows. Let $\Phi: \Sigma \to \bar{M} = M_1 \times M_2$ be an isometric immersion, where $\Sigma$ is an $n$-dimensional compact Riemannian manifold, $M_1$ is  an $m_1$-dimensional compact Riemannian submanifold, and $M_2$ is an $m_2$-dimensional Riemannian manifold. Let $d = m_1 + m_2 - n$ denote the codimension of $\Sigma$ in $M_1 \times M_2$.
We further let $f_1: M_1 \to \R^{N_1}$ be an isometric embedding of $M_1$ in an Euclidean space, and let $F: M_1 \times M_2 \to \R^{N_1} \times M_2$ be defined by $F = f_1 \times \Id_{M_2}$. Let $B_1$ denote the second fundamental form of the immersion $f_1$. We can then consider the chain of immersions $\Sigma \xrightarrow{\Phi} M_1 \times M_2 \xrightarrow{F}  \R^{N_1} \times M_2$.

Let $p \in \Sigma$ be a point. We choose a local orthonormal frame $\{e_1, \dots, e_n, \eta_1, \dots, \eta_{d}\}$ near $\Phi(p) \in M_1 \times M_2$ such
that $\{e_1,\dots, e_n\}$ is an orthonormal frame in $d\Phi(T\Sigma)$ and $\{\eta_{1}, \dots, \eta_{d}\}$ is an orthonormal frame in $N\Sigma \subset T(M_1 \times M_2)$. For any tangent vector $v$ at $\Phi(p) = (p_1, p_2) \in M_1 \times M_2$, we decompose $v$ as $v = (v^1, v^2)$, where $v^i$ is tangent to
$M_i$ at $p_i$. We will thus write $e_i = (e_i^1, e_i^2)$ and $\eta_\beta = (\eta_\beta^1, \eta_\beta^2)$.


For a fixed vector $U \in \R^{N_1}$, by identifying $U$ with $(U, 0) \in T_{F(\Phi(p))} (\R^{N_1} \times M_2)$, we can decompose $U$ as 
$$U = U^T + U^N,$$
where $U^T \in dF(T_{\Phi(p)}(M_1 \times M_2))$, and $U^N \in N_{\Phi(p)}(M_1 \times M_2)$. For $U^T \in dF(T_{\Phi(p)}(M_1 \times M_2))$, by identifying it with its preimage in $T_{\Phi(p)}(M_1 \times M_2)$ and considering the immersion $\Phi: \Sigma \to M_1 \times M_2$, we can further decompose it as 
$$U^T = T_U + N_U,$$
where $T_U \in d\Phi(T_p \Sigma)$ and $N_U \in N_p \Sigma \subset T_{\Phi(p)}(M_1 \times M_2)$.

The case that will be interesting to us is when $\Sigma$ is a stable minimal submanifold of $M_1 \times M_2$. To study this case, we need to use some appropriate normal sections in the second variation formula. To this end, let $\{E_1,\dots, E_{N_1}\}$ be an orthonormal basis of $\R^{N_1}$. Then $N_{E_{N_A}}$, $A= 1, \dots, N_1$ will be the normal sections we use. Computation shows that
\begin{lem}\cite[Equation (2.8)]{chen2013stable} \label{lem: 2FF computation}
$$\sum_{A=1}^{N_1} -\langle  N_{E_A},J_{\Sigma}(N_{E_A})\rangle = \sum_{i=1}^{n} \sum_{\beta =1}^d \Big( \langle B_1(\eta_\beta^1, \eta_\beta^1),B_1(e_i^1, e_i^1) \rangle - 2\|B_1(e_i^1,\eta_\beta^1)\|^2 \Big).$$
\end{lem}

In the next three subsections, we will let $M_1$ be $\C P^{m_1/2}$, $\h P^{m_1/4}$, and $\O P^{2}$ respectively, and study the stable compact minimal immersions $\Sigma \to M_1 \times M_2$. For these projective spaces, there exist isotropic embeddings into Euclidean space so that $\|B_1(X,X)\|$ is constant for any unit vector $X \in TM_1$. In all these cases, we will show $\sum_{A=1}^{N_1} -\langle  N_{E_A},J_{\Sigma}(N_{E_A})\rangle \equiv 0$, and therefore obtain geometric information about $\Sigma$.

\subsection{Stable minimal submanifolds in \texorpdfstring{$\mathbb{C}P^{\frac{m_{1}}{2}}\times \mathbb{C}P^{\frac{m_{2}}{2}}$}{the product of two complex projective spaces}.}

We equip each $\C P^{m_k/2}$, $k=1,2$ with the Fubini-Study metric of constant
holomorphic sectional curvature $\lambda_k^2 : = \frac{2m_{k}}{m_k+2}$. 

Let $\Phi = (\phi_1, \phi_2) \colon \Sigma \to \bar{M} = \C P^{m_1/2} \times M_2$ be a stable compact minimal immersion of
codimension $d$ and dimension $n$.

Let $f_{1}:\mathbb{C}P^{\frac{m_{1}}{2}}\rightarrow \mathbb{R}^{\tilde{m}_1}$ be the isotropic embedding induced by the generalized Veronese embedding (see Section 2 of \cite{sakamoto1977planar}). The nice property about this embedding is that, if we let $B$ denote its second fundamental form, then at each $q\in \mathbb{C}P^{\frac{m_{1}}{2}}$, we have
$$\|B(X, X)\|^2 = \lambda_1^2 \text{ for any unit vector $X \in T_{q}\mathbb{C}P^{\frac{m_{1}}{2}}$}$$
(see Proposition 3.6 of \cite{ohnita1986stable}).

For each $v\in \mathbb{R}^{\tilde{m}_1}$, we identify $v$ with $(v, 0) \in T (\mathbb{R}^{\tilde{m}_1}\times M_2)$. Same as in Section \ref{subsubsec: imm}, we denote by $N_v$  the projection of $v$ onto the normal space $N_{p}\Sigma$ in $T_{\Phi(p)}\bar{M}$. We are now going to use the normal sections $N_v$ in the second variation formula, where $v$ runs over an orthonormal basis of the Euclidean space $\R^{\tilde{m}_1}$.

Let $p \in \Sigma$. Let $\{e_1, \dots, e_n\}$ be an orthonormal basis of $d\Phi(T_p\Sigma)$. Let $e_j^1$ be the projection of $e_j$ onto $T_{\phi_1(p)}\C P^{m_1/2}$ and $e_j^2$ be the projection of $e_j$ onto $T_{\phi_2(p)}M_2$. Let $J_1$ be the complex structure of $T_{\phi_1(p)}\C P^{m_1/2}$. 
\begin{lem}\cite[Lemma~3.1]{Ramirez3}
\label{GENERALC} Same assumptions as above. Let $\{E_{1},\ldots,E_{\tilde{m}_1}\}$ be the usual canonical basis of $\mathbb{R}^{\tilde{m}_1}$. Then

\begin{equation}
\label{MAIN}
\displaystyle\sum_{A=1}^{\tilde{m}_1} -\langle  N_{E_A},J_{\Sigma}(N_{E_A})\rangle = \lambda_1^{2} \Big (  \sum_{i,j=1}^{n}
\langle  J_1(e^{1}_{j}),e_{i}^{1}\rangle^{2}-\langle  e^{1}_{j},e^{1}_{i}\rangle^{2}\Big ).
\end{equation}
\end{lem}

\color{black}
The stability of $\Phi$ requires

\begin{equation*}
    \int_\Sigma  -\langle  N_{E_A},J_{\Sigma}(N_{E_A})\rangle \, d\Sigma \ge 0 \quad \text{for each } A=1, \dots, \tilde{m}_1,
\end{equation*}
so
\begin{equation*}\label{ineq: stability, C}
    \int_\Sigma \sum_{A=1}^{\tilde{m}_1} -\langle  N_{E_A},J_{\Sigma}(N_{E_A})\rangle \, d\Sigma \ge 0.
\end{equation*}
On the other hand, combining Lemma \ref{GENERALC} and Proposition \ref{ineq: key result}, we have
\begin{equation*}
\label{sign}
\displaystyle \sum_{A=1}^{\tilde{m}_1} -\langle  N_{E_A},J_{\Sigma}(N_{E_A})\rangle = \lambda_1^{2}  \sum_{i,j=1}^n \langle  J_1(e_j^1), e_i^1\rangle^2 - \langle  e_j^1, e_i^1\rangle^2 \le 0.
\end{equation*}
The above two equations together show that the integrand must be exactly zero, namely,
\begin{equation}\label{eqn: complex, =0, J_1}
\sum_{i,j=1}^n \langle  J_1(e_j^1), e_i^1\rangle^2 - \langle  e_j^1, e_i^1\rangle^2 = 0.
\end{equation}

Now we specialize to the case where $M_2 = \C P^{m_2/2}$, and consider the stable compact minimal immersion $\Phi = (\phi_1, \phi_2) \colon \Sigma \to \C P^{m_1/2} \times \C P^{m_2/2}$ of dimension $n$. For $p \in \Sigma$, let $J_2$ be the complex structure of $T_{\phi_2(p)}\C P^{m_2/2}$. Then letting $\C P^{m_2/2}$ play the role of $\C P^{m_1/2}$ in the above arguments, we have 
\begin{equation}\label{eqn: complex, =0, J_2}
\sum_{i,j=1}^n \langle  J_2(e_j^2), e_i^2\rangle^2 - \langle  e_j^2, e_i^2\rangle^2 = 0.
\end{equation}
as well.

 For $k=1,2$, choose a local orthonormal basis $\{v_i^k\}_{i=1}^{m_k}$ for $T_{\phi_k(p)}\C P^{m_k/2}$ and write $e_j^k$ under this basis as $e_j^k = b^k_{i,j}v_i$. Let $X_k$ denote the $m_k \times n$ matrix $X_k = \begin{bmatrix}b^k_{i,j} \end{bmatrix}_{i,j}$. Then $X_1^TX_1 = \begin{bmatrix}\langle  e_i^1, e_j^1\rangle_{i,j=1}^n \end{bmatrix}$ and $X_2^TX_2 = \begin{bmatrix}\langle  e_i^2, e_j^2\rangle_{i,j=1}^n \end{bmatrix}$,
and since $\{(e_1^1, e_1^2), \dots, (e_n^1, e_n^2)\}$ forms an orthonormal basis of $d\Phi(T_p\Sigma)$ we have 

$$X_1^TX_1 + X_2^TX_2 = I_n.$$

Notice that $X_1^TX_1$ and $X_2^TX_2$ are positive semi-definite since they are Gram matrices. Moreover, if $\mu$ is an eigenvalue of $X_1^TX_1$ with eigenvector $v$, then we have $X_1^TX_1v + X_2^TX_2v = v$, so $X_2^TX_2v = v - X_1^TX_1v = v - \mu v = (1-\mu)v$, showing $(1-\mu)$ is an eigenvalue of $X_2^TX_2$ with eigenvector $v$.\\
Thus if $1 \ge \mu_1 \ge \dots \ge \mu_n \ge 0$ are eigenvalues of $X_1^TX_1$ in non-increasing order with corresponding orthonormal eigenbasis $v_1, \dots, v_n$, then $0 \le 1-\mu_1 \le \dots \le 1-\mu_n \le 1$ are eigenvalues of $X_2^TX_2$ in non-decreasing order with corresponding orthonormal eigenbasis $v_1, \dots, v_n$. In particular, the multiplicity of $1$ as an eigenvalue of $X_1^TX_1$ equals the multiplicity of $0$ as an eigenvalue of $X_2^TX_2$.\\
 
By Proposition \ref{ineq: key result}, Equation (\ref{eqn: complex, =0, J_1}) implies that $X_1X_1^T$ commutes with $J_{1}$, and Equation (\ref{eqn: complex, =0, J_2}) implies that $X_2X_2^T$ commutes with $J_{2}$. Thus the eigenspaces of $X_1X_1^T$ are $J_{1}$-invariant, therefore even dimensional. Thus all the eigenvalues of $X_1X_1^T$ have even multiplicity. By Proposition \ref{linearalgebra}, this shows all the nonzero eigenvalues of $X_1^TX_1$ have even multiplicity. Similarly, all the nonzero eigenvalues of $X_2^TX_2$ have even multiplicity.\\
In particular, $1$ as an eigenvalue of $X_2^TX_2$ has even multiplicity. Using the fact that the multiplicity of $1$ as an eigenvalue of $X_1^TX_1$ equals the multiplicity of $0$ as an eigenvalue of $X_2^TX_2$, we get that the multiplicity of $0$ as an eigenvalue of $X_1^TX_1$ is also even. Thus all the eigenvalues of $X_1^TX_1$ have even multiplicity. However, sum of the multiplicities of the eigenvalues of $X_1^TX_1$ equals its dimension, which is $n$. This shows $n$ must be an even number.

 We have thus proved Theorem \ref{main1}.
\begin{rem}
\label{cc:even}
Notice that by Proposition \ref{ineq: key result}, from Equations (\ref{eqn: complex, =0, J_1}) and (\ref{eqn: complex, =0, J_2}) we can deduce that $\Span\{e^{k}_{1}, \dots, e^{k}_{n}\}$ is $J_{k}$-invariant, for $k=1,2$. However, this is not enough to conclude that the tangent space of the stable submanifold $\Sigma$ behaves well under a complex structure of $\C P^{m_1/2} \times \C P^{m_2/2}$.
\end{rem}

\subsection{Stable minimal submanifolds in \texorpdfstring{$\mathbb{H}P^{\frac{m_{1}}{4}}\times M$}{the product of a quaternionic projective space with any Riemannian manifold}.} 

Equip $\mathbb{H}P^{\frac{m_{1}}{4}}$ with its standard metric as a Riemannian symmetric space, with the
maximum $\lambda^2$ of the sectional curvatures given by $\lambda^2 : = \frac{2m_{1}}{m_1+4}$. 
Let $\Phi=(\phi_1,\phi_2):\Sigma\rightarrow \bar{M}:= \mathbb{H}P^{\frac{m_{1}}{4}}\times M$ be a stable compact minimal immersion of codimension $d$ and dimension $n$, where $M$ is any Riemannian manifold of dimension $m_{2}$.

Let $f_1:\mathbb{H}P^{\frac{m_{1}}{4}}\rightarrow \mathbb{R}^{m}$ be the isotropic embedding induced by the generalized Veronese embedding (see Section 2 of \cite{sakamoto1977planar}). The nice property about this embedding is that, if we let $B$ denote its the second fundamental form, then at each $q\in \mathbb{H}P^{\frac{m_{1}}{4}}$, we have
$$\|B(X, X)\|^2 = \lambda^2 \text{ for any unit vector $X \in T_{q}\mathbb{H}P^{\frac{m_{1}}{4}}$}$$
(see Proposition 3.6 of \cite{ohnita1986stable}).

For each $v\in \mathbb{R}^{m}$, we identify $v$ with $(v, 0) \in T (\mathbb{R}^{m}\times M)$. 
Same as in Section \ref{subsubsec: imm}, we denote by $N_v$  the projection of $v$ onto the normal space $N_{p}\Sigma$ in $T_{\Phi(p)}\bar{M}$. As in the complex case, we are going to use the normal sections $N_v$ in the second variation formula, where $v$ runs over an orthonormal basis of  $\mathbb{R}^{m}$.

Let $p\in \Sigma$. Let $\{e_{1},\ldots,e_{n}\}$
be an orthonormal basis of $d\Phi(T_{p}\Sigma)$, and let $\{\eta_{1},\ldots,\eta_{d}\}$ be an orthonormal basis of $N_{p}\Sigma$. Let $e_j^1$ be the projection of $e_j$ onto $T_{\phi_1(p)}\h P^{m_1/4}$ and $e_j^2$ be the projection of $e_j$ onto $T_{\phi_2(p)}M$. Similarly, let $\eta_j^1$ be the projection of $\eta_j$ onto $T_{\phi_1(p)}\h P^{m_1/4}$ and $\eta_j^2$ be the projection of $\eta_j$ onto $T_{\phi_2(p)}M$.
Let $\{J_1,J_2,J_3\}$ be the quaternionic structure on $T_{\phi_1(p)}\h P^{m_1/4}$. 

\begin{lem}\cite[Lemma~4.1]{Ramirez3}
\label{GENERALH}
Same assumptions as above. Let $\{E_{1},\ldots,E_{m}\}$ be the usual canonical basis of $\mathbb{R}^{m}$. Then, for $s\in \{1,2,3\}$

\begin{equation}
    \label{hMAIN2} \sum_{A=1}^{m} -\langle N_{E_{A}},J_{\Sigma}(N_{E_{A}})\rangle = \lambda^{2}\Big(-\sum_{
         k\neq s}^{3}\sum_{j=1}^{n}\sum_{\beta=1}^{d} \langle  J_{k}(\eta^{1}_{\beta}),e^{1}_{j}\rangle^{2}
+\sum_{i,j=1}^{n}(\langle  J_{s}(e^{1}_{i}),e_{j}^{1}\rangle^{2}-\langle  e^{1}_{i},e_{j}^{1}\rangle^{2})\Big).
\end{equation}
\end{lem}

The stability of $\Phi$ requires
\begin{equation*}
    \int_\Sigma -\langle N_{E_{A}},J_{\Sigma}(N_{E_{A}})\rangle \, d\Sigma \ge 0 \quad \text{for each } A=1, \dots, m,
\end{equation*}
so
\begin{equation}\label{ineq: stability, H}
    \int_\Sigma \sum_{A=1}^{m} -\langle N_{E_{A}},J_{\Sigma}(N_{E_{A}})\rangle \, d\Sigma \ge 0.
\end{equation} 
On the other hand, combining Lemma \ref{GENERALH} and Proposition \ref{ineq: key result}, we have for $s=1,2,3$

\begin{equation}
\label{sign2}
\begin{aligned}
 & \sum_{A=1}^{m} -\langle N_{E_{A}},J_{\Sigma}(N_{E_{A}})\rangle \\
= & \, \lambda^{2}\Big(-\sum_{
         k\neq s}^{3}\sum_{j=1}^{n}\sum_{\beta=1}^{d} \langle  J_{k}(\eta^{1}_{\beta}),e^{1}_{j}\rangle^{2}
+\sum_{i,j=1}^{n}(\langle  J_{s}(e^{1}_{i}),e_{j}^{1}\rangle^{2}-\langle  e^{1}_{i},e_{j}^{1}\rangle^{2})\Big) \\
\le &  \, \lambda^{2} \sum_{i,j=1}^{n} (\langle  J_{s}(e^{1}_{i}),e_{j}^{1}\rangle^{2}-\langle  e^{1}_{i},e_{j}^{1}\rangle^{2}) \\
\le & \, 0.
\end{aligned}
\end{equation}

Equations (\ref{ineq: stability, H}) and (\ref{sign2}) 
together show that both inequalities in (\ref{sign2}) are equalities. Hence for $s=1,2,3$,
\begin{equation}
\label{quaternionic2}
\sum_{i,j=1}^n \langle  J_{s}(e_j^1), e_i^1\rangle^2 - \langle  e_j^1, e_i^1\rangle^2 = 0,
\end{equation}
and consequently
\begin{equation}
\label{quaternionic}
 -\sum_{k\neq s}^{3}\sum_{j=1}^{n}\sum_{\beta=1}^{d} \langle  J_{k}(\eta^{1}_{\beta}),e^{1}_{j} \rangle^{2} = 0.
\end{equation}

From Equation (\ref{quaternionic}) we have that for $t\in \{1,2,3\}$,  $\beta\in \{1,\dots,d\}$, and $j\in \{1,\dots,n\}$
\begin{equation*}
\langle  J_t(\eta_\beta ^1), e_j^1 \rangle=0.
\end{equation*}
Thus for $t, \beta$ fixed, $J_t(\eta_\beta ^1) \perp \Span\{e_l^1\}_l$. By Proposition \ref{ineq: key result}, Equation (\ref{quaternionic2}) also implies that $\Span\{e_l^1\}_l$ is $J_t$-invariant, so $J_t(e_j^1) \in \Span\{e_l^1\}_l$ for each $j$. Thus 
\begin{equation}
\label{E1}
\langle  \eta_\beta ^1, e_j^1 \rangle = \langle  J_t(\eta_\beta ^1), J_t(e_j^1) \rangle = 0.
\end{equation}
But since we have $0 = \langle  \eta_\beta, e_j \rangle = \langle  \eta_\beta ^1, e_j^1 \rangle + \langle  \eta_\beta ^2, e_j^2 \rangle$, Equation (\ref{E1}) shows that we necessarily have
\begin{equation}
\label{E2}
    \langle  \eta_\beta ^2, e_j^2 \rangle =  0
\end{equation}
as well.\\

For the product manifold $\h P^{m_1/4} \times M$, let 
$$P: T(\h P^{m_1/4}\times M)  \to T\h P^{m_1/4}  \oplus 0 \subseteq T(\h P^{m_1/4}\times M),$$
$$Q: T(\h P^{m_1/4}\times M)\to 0 \oplus TM \subseteq T(\h P^{m_1/4}\times M)$$
denote the projection maps. Let $F = P-Q$, then $F^2 = I$. For $p \in \Sigma$, we claim that 
$F( d\Phi(T_{p}{\Sigma})) \subseteq d\Phi(T_{p}{\Sigma})$.
In fact, since $F(e_i) = F(e_i^1, e_i^2) =  (e_i^1, -e_i^2)$, for each $\eta_\beta$, we have
$$
 \langle  F(e_i), \eta_\beta\rangle   = \langle  e_i^1, \eta_\beta^1\rangle -  \langle  e_i^2, \eta_\beta^2\rangle = 0,
$$
where we have used Equations (\ref{E1}) and (\ref{E2}) in the second equality. Therefore, $F(e_i) \perp (d\Phi(T_p\Sigma))^\perp$. This shows $F(e_i) \subset d\Phi(T_p\Sigma)$, so $F( d\Phi(T_{p}{\Sigma})) \subseteq d\Phi(T_{p}{\Sigma})$. Thus $\Sigma$ is an invariant submanifold of $\h P^{m_1/4} \times M$.

By Theorem 1 in \cite{xu2000submanifolds}, $\Sigma$ is isometric to a product manifold $\Sigma_1 \times \Sigma_2$, where $M_1$ is an immersion into $\h P^{m_1/4}$ and $\Sigma_2$ is an immersion into $M$. Since $\Phi:\Sigma \to \h P^{m_1/4} \times M$ is a stable compact minimal immersion, we further have that $\Sigma_1$ and $\Sigma_2$ are stable compact minimal immersions into $\h P^{m_1/4}$ and $M$, respectively. However, Ohnita \cite[Theorem~D]{ohnita1986stable} showed that the only stable minimal immersed submanifolds of a quaternionic protective subspace are precisely the quaternionic protective subspaces. Using this result, we have thus established Theorem \ref{main2}.

 \begin{rem}
 Notice that the diagonal map $\Phi:\h P^{m/4} \to \h P^{m/4} \times \h P^{m/4}, x\mapsto (x,x)$ is not stable.
 \end{rem}

\subsection{Stable minimal submanifolds in \texorpdfstring{$\mathbb{O}P^2\times M$}{the product of the octonionic projective plane with any Riemannian manifold}.}

Let $\mathbb{O}P^2$ denote the octonionic (Cayley) projective plane, endowed with the standard metric as a Riemannian symmetric space. Classically, $\mathbb{O}P^2$ can be seen as a 16-dimensional quotient manifold $F_4/\Spin(9)$, where the isometry group of $\O P^2$ is isomophic to the exceptional Lie group $F_4$ and the isotropy group at any point $q \in \O P^2$ is isomorphic to $\Spin(9)$.

Alternatively, we can also define $\O P^2$ in terms of equivalence classes (cf. \cite{held2009semi}). Consider the relation $\sim$ on $\O^3$, where $[a, b, c] \sim [d, e, f]$ if and only if there exists $u \in \O \setminus \{0\}$ such that
$a = du, b = eu, c = fu$. This relation is symmetric and reflexive. However, it is not necessarily transitive due to non-associativity of octonions. As a remedy, we instead consider the following subsets of $\O^3$:
$$U_1 := \{1\} \times \O \times \O, \quad U_2 := \O \times \{1\}  \times \O, \quad U_3 := \O   \times \O \times \{1\},$$
and their union $\O^3_\bullet := U_1 \cup U_2 \cup U_3$. The relation $\sim$ on $\O^3_\bullet$ is then an equivalence relation, and we can define the octonionic projective plane $\O P^2$ as the set of equivalence classes of $\O^3_\bullet$ by $\sim$. That is, 
$$\O P^2 := \O^3_\bullet / \sim.$$

\subsubsection{Basic octonion algebra}
First we review some basic octonion algebra. We refer the reader to \cite{brown1972riemannian}, \cite{held2009semi}, and \cite{kotrbaty2020integral} for details. Let $I_0=1, I_1, \dots, I_7$ denote the usual octonion basis.
For $a = a_0 + \sum_{s=1}^7 a_s I_s \in \O$, $a_i \in \R$, let
\begin{center}
$a^* : = a_0 - \sum_{s=1}^7 a_s I_s$
\end{center}
 denote its conjugate. Then the inner product on $\O$ is given by 
$$\langle a, b \rangle = \frac{1}{2}(a^*b + b^*a) = \frac{1}{2}(ba^* + ab^*).$$

Moreover, for $a,b,c,d, x,y \in \O$, we have the identities
\begin{align} 
\label{id: octonion1}\langle a x, y\rangle &= \langle x, a^*y \rangle, \\   
\label{id: octonion2} \langle xa, y\rangle &= \langle x, ya^* \rangle, \\
\label{id: octonion3} \langle ab, cd \rangle + \langle ad, cb \rangle &= 2 \langle a, c \rangle\langle b, d \rangle.
\end{align}

The 16-dimensional real vector space $\O \oplus \O$ decomposes into octonionic lines of the form
$$\ell_m := \{(u,mu) \mid u \in \O\} \quad \text{or} \quad \ell_\infty := \{(0,u) \mid u \in \O\},$$
that intersect each other only at the origin $(0, 0) \in \O \oplus \O$. Each line is an 8-dimensional real vector space. Here $m \in \O \cup \{\infty\} \cong  \O P^1 \cong S^8 $ parametrizes the set of octonionic lines. 
Be cautious here that the octonionic line through $(0,0)$ and $(a,b)$ is \emph{not} the set $\{(au,bu) \mid u \in \O\}$ due to the non-associativity of octonion multiplication; the correct line is given by $\ell_{ba^{-1}}$.

\begin{prop}[see e.g.,{\cite[Section~2.2]{kotrbaty2020integral}}]\label{prop: Spin(9)}
Here we collect the properties of $\Spin(9)$ that we will use later.
 \begin{itemize}
 \item The group $\Spin(9)$ acts transitively and effectively on the unit sphere $S^{15} \subset \O \oplus \O$, and we can view $\Spin(9) \subset SO(16)$ as a matrix group acting on $\O \oplus \O$.
 \item $\Spin(9)$ maps octonionic lines to octonionic lines.
 \item The setwise stabilizer of an octonionic line is isomorphic to $\Spin(8)$, and the stabilizer of a point is isomorphic to $\Spin(7)$.
 \end{itemize}   
\end{prop}
 
Define the map $L: \O \oplus \O \setminus 0 \to \O P^1$ by mapping $x \in \O \oplus \O$ to the unique octonionic line $L(x)$ that contains $x$. 
The restriction of the map $L$ to $S^{15} \subset \O \oplus \O$ defines the \emph{octonionic Hopf fibration} 
$$S^{15} \to S^8,  \quad \text{or as homogeneous fibration} \quad \Spin(9)/\Spin(7) \to \Spin(9)/\Spin(8)$$
with $S^7 \cong \Spin(8)/\Spin(7)$ as the fiber. $\Spin(9)$ is the symmetry group of this fibration. (cf. \cite{ornea2013spin}) 

\begin{lem}\label{lem: octo line inner product}
Let $\ell_{m_1}, \ell_{m_2}$, $m_1, m_2 \in\O \cup \{\infty\}$ be two octonionic lines defined as above. Let $x_1, \dots, x_8$ be an orthonormal basis of $\ell_{m_1}$ and $y_1, \dots, y_8$ be an orthonormal basis of $\ell_{m_2}$. 
Let $X$, $Y$ be the associated $16 \times 8$ matrices
$$X = \begin{bmatrix}x_1 & \dots & x_8 \end{bmatrix}, \quad Y = \begin{bmatrix}y_1 & \dots & y_8 \end{bmatrix}.$$ 
Then $X^TY  =  [\langle x_i, y_j \rangle]_{i,j=1}^8 = cQ$, where $Q$ is an $8 \times 8$ orthogonal matrix, and $0 \le c \le 1$ is some constant. We have $c = 0$ if and only if $\ell_{m_1} \perp \ell_{m_2}$, and $c = 1$ if and only if $m_1 = m_2$.
\end{lem}
\begin{proof}
By Proposition \ref{prop: Spin(9)}, by an action of $\Spin(9)$ we can assume that $m_1 = 0$ and consequently $\ell_{m_1} = \O \oplus 0$. For $m_2 =\infty$, since $\ell_0, \ell_\infty$ are orthogonal complements, we have $X^TY = 0 = cQ$, for $c=0$ and $Q = \Id_8$. We can thus assume $m_2  \in \O$.

Let $\tilde{x}_s = (I_{s-1}, 0)$, $s=1,\dots, 8$, where $I_0=1, I_1, \dots, I_7$ are the usual octonion basis. Then $\tilde{x}_1, \dots, \tilde{x}_8$ form an orthogonal basis for the 8-dimensional vector space $\ell_0$. Let 
$$\tilde{X} = \begin{bmatrix}\tilde{x}_1 & \dots & \tilde{x}_8 \end{bmatrix}.$$
Therefore, there exists an $8 \times 8$ orthogonal matrix $Q_1$ such that $X = \tilde{X}Q_1$.

Let $\tilde{y}_s = \frac{1}{\sqrt{1+|m_2|^2}}(I_{s-1}, m_2I_{s-1})$, $s=1,\dots, 8$. Then $\tilde{y}_1, \dots, \tilde{y}_8$ form an orthogonal basis for the 8-dimensional vector space $\ell_{m_2}$. Let 
$$\tilde{Y} = \begin{bmatrix}\tilde{y}_1 & \dots & \tilde{y}_8 \end{bmatrix}.$$
Therefore, there exists an $8 \times 8$ orthogonal matrix $Q_2$ such that $Y = \tilde{Y}Q_2$.

It is clear to see that $$\tilde{X}^T\tilde{Y} = [\langle \tilde{x}_i, \tilde{y}_j \rangle]_{i,j=1}^n = \frac{1}{\sqrt{1+|m_2|^2}}\Id_8.$$
Thus
$$X^TY = Q_1^T \tilde{X}^T\tilde{Y}Q_2 = Q_1^T(\frac{1}{\sqrt{1+|m_2|^2}}\Id_8)Q_2 = \frac{1}{\sqrt{1+|m_2|^2}}(Q_1^TQ_2),$$ 
where $Q: = Q_1^TQ_2$ is an $8 \times 8$ orthogonal matrix, and $c = \frac{1}{\sqrt{1+|m_2|^2}} \in (0,1]$. We have $c = 0$ precisely when $m_2=0$ and consequently $\ell_{m_1} = \ell_{m_2}$.

\end{proof}

\subsubsection{Geometry of $\O P^2$}
Here we review the basic Riemannian geometry of $\O P^2$. We refer the reader to \cite{brown1972riemannian}, \cite{ohnita1986stable}, and \cite{held2009semi} for details. In each of the chart $U_1, U_2, U_3$ defined at the beginning of this subsection, let $(u,v)$ with $u,v \in \O$ denote the coordinate functions on the chart. Then they give rise to octonion valued
1-forms $du$ and $dv$. The standard metric of $\O P^2$ is then given by (see Equation (3.1) in \cite{held2009semi})
$$ ds^2 = \frac{4}{\lambda^2}\frac{|du|^2(1 + |v|^2) + |dv|^2(1 + |u|^2) - 2\Re[(u v^*)(dvdu^*)]}{
(1 + |u|^2 + |v|^2)^2},$$
where $\lambda^2$ is a scaling factor that denotes the maximum of the sectional curvatures of $\mathbb{O}P^{2}$. Notice that this expression resembles the expression of the Fubini–Study metric of a complex projective plane.

Let $R$ denote the Riemannian curvature tensor of $\mathbb{O}P^{2}$. Let $q$ be a point in $\mathbb{O}P^{2}$. We can identify $T_q(\mathbb{O}P^{2})$ with $\O \oplus \O$ in a natural manner. Using the structure of the octonionic algebra, the curvature tensor $R$ is given by 

\begin{align}
\langle  R\Big((a,b),(c,d)\Big)(e,f), (g,h)\rangle 
  = \frac{\lambda^2}{4}\bigg( &-4\langle a,e\rangle\langle c,g\rangle + 4\langle c,e\rangle\langle a,g\rangle \\ 
& -4\langle b,f\rangle\langle d,h\rangle + 4\langle d,f\rangle\langle b,h\rangle  \nonumber \\ 
& + \langle ed^*,gb^*\rangle - \langle eb^*,gd^*\rangle + \langle cf^*, ah^*\rangle - \langle af^*,ch^*\rangle \nonumber\\ 
& + \langle ad^*-cb^*,gf^*-eh^*\rangle \bigg),    \nonumber
\end{align}
where $a,b,c,d,e,f,g,h \in \O$.
For this formula, see \cite[Theorem 5.3 (1)]{held2009semi}, or see \cite[Section~1.2 and Table 2]{ohnita1986stable} and apply the basic identities (\ref{id: octonion1}) and (\ref{id: octonion2}). Notice that we use a slightly different identification for $T_p(\mathbb{O}P^{2})$ from the one in \cite{ohnita1986stable} ($\O \oplus \O$ versus $\O \oplus \overline{\O}$) and there is a typo in Equation (1.3) of \cite{ohnita1986stable}, where $v$ at the end of the second line should be $v^*$. Here we use this identification so as to be consistent with out definition of octonionic lines earlier. 

In particular, we have 
\begin{equation}\label{eqn: curvature}
\begin{aligned}
\langle  R\Big((a,b),(c,d)\Big)(a,b), (c,d)\rangle 
  = \frac{\lambda^2}{4}\bigg( &-4\langle a,a\rangle\langle c,c\rangle + 4\langle a,c\rangle^2  -4\langle b,b\rangle\langle d,d\rangle + 4\langle b,d\rangle^2\\
& + 2\langle ad^*,cb^*\rangle - 2\langle ab^*,cd^*\rangle - \langle ad^*-cb^*,ad^*-cb^*\rangle \bigg). 
\end{aligned}
\end{equation}

Let $f_1:\mathbb{O}P^{2}\rightarrow \mathbb{R}^{m}$ be the isotropic embedding induced by the generalized Veronese embedding (see Section 2 of \cite{sakamoto1977planar}). Let $B$ denote the second fundamental form of this embedding. Then at each $q\in \mathbb{\O} P^{2}$, we have
$$\|B(X, X)\|^2 = \lambda^2 \text{ for any unit vector $X$ in $T_{q}\O P^2$}.$$
Applying the Gauss equation yields
\begin{equation}
\begin{aligned}
\label{eqn: 2FF}
   \displaystyle 3\langle B(X,Y),B(Z,W)\rangle=&\langle R(X,Z)W,Y \rangle+\langle R(X,W)Z,Y\rangle +\lambda^{2}\langle X,Y\rangle \langle Z,W\rangle\\
   \displaystyle &+\lambda^{2} \langle X,W\rangle \langle Y,Z\rangle+\lambda^{2}\langle X,Z\rangle \langle W,Y\rangle,
   \end{aligned}
\end{equation}
where  $X,Y,Z,W \in T_{q}\mathbb{\O} P^{2}$ (see Proposition 3.6 in \cite{ohnita1986stable}).

\subsubsection{Proof of Theorem \ref{main3}}

Let $\Phi=(\phi_1,\phi_2):\Sigma\rightarrow \bar{M}:= \mathbb{O}P^{2}\times M$ be a stable compact minimal immersion of codimension $d$ and dimension $n$, where $M$ is any Riemannian manifold of dimension $m_{2}$.

Let $p\in \Sigma$, $\{e_{1},\ldots,e_{n}\}$
be an orthonormal basis of $d\Phi(T_{p}\Sigma)$, and  $\{\eta_{1},\ldots,\eta_{d}\}$ be an orthonormal basis of $N_{p}\Sigma$. Let $e_j^1$ be the projection of $e_j$ onto $T_{\phi_1(p)}\O P^{2}$ and $e_j^2$ be the projection of $e_j$ onto $T_{\phi_2(p)}M$. Similarly, let $\eta_j^1$ be the projection of $\eta_j$ onto $T_{\phi_1(p)}\O P^{2}$ and $\eta_j^2$ be the projection of $\eta_j$ onto $T_{\phi_2(p)}M$. 
We identify $T_{\phi_1(p)}\O P^{2}$ with $\O \oplus \O$.



For each $v\in \mathbb{R}^{m}$, we identify $v$ with $(v, 0) \in T (\mathbb{R}^{m}\times M)$. 
Same as in Section \ref{subsubsec: imm}, we denote by $N_v$  the projection of $v$ onto the normal space $N_{p}\Sigma$ in $T_{\Phi(p)}\bar{M}$. As in the complex and quaternionic cases, we are going to use the normal sections $N_v$ in the second variation formula, where $v$ runs over an orthonormal basis of  $\mathbb{R}^{m}$.

\begin{lem}
\label{GENERALO}
Same assumptions as above. Let $\{E_{1},\ldots,E_{m}\}$ be the usual canonical basis of $\mathbb{R}^{m}$. Then, 
\begin{align}
& \sum_{A=1}^{m} -\langle N_{E_{A}},J_{\Sigma}(N_{E_{A}})\rangle \nonumber \\
=& {\lambda^2} \bigg(\sum_{j=1}^n\sum_{i=1}^n |e_j^1|^2|\Proj_{L(e_j^1)} e_i^1|^2  - 8\sum_{j=1}^n\sum_{i=1}^n\langle e^{1}_{j},e_i^1\rangle^2 - 6\sum_{j=1}^n\sum_{k=1}^d \langle e_j^1, \eta_k^1 \rangle^2  \bigg) \label{oMAIN1} \\
=& {\lambda^2} \bigg(\sum_{k=1}^d\sum_{l=1}^d |\eta_k^1|^2|\Proj_{L(\eta_k^1)} \eta_l^1|^2  - 8\sum_{k=1}^d\sum_{l=1}^d\langle \eta^{1}_{k},\eta^{1}_{l}\rangle^2- 6\sum_{j=1}^n\sum_{k=1}^d \langle e_j^1, \eta_k^1 \rangle^2  \bigg), \label{oMAIN2}
\end{align}
where $L(e_j^1)$ is the octonionic line containing $e_j^1$ as defined above when $e_j^1 \neq 0$. When $e_j^1 = 0$, we just let $L(e_j^1)$ be any octonionic line. 
Same for $L(\eta_k^1)$.
\end{lem}

\begin{proof}
We proceed in the same way as in the proof of Lemma 3.1 in \cite{Ramirez3}. By Lemma \ref{lem: 2FF computation}, we have the following

\begin{equation}
\label{mainequation}
    \sum_{A=1}^{m} -\langle N_{E_{A}},J_{\Sigma}(N_{E_{A}})\rangle=\sum_{j=1}^{n} \sum_{k=1}^{d}2|B(e_{j}^{1},\eta_{k}^{1})|^{2}-\langle B(\eta_{k}^{1},\eta_{k}^{1}),B(e_{j}^{1},e_{j}^{1})\rangle.
\end{equation}
Using Equation (\ref{eqn: 2FF}), we have
\begin{center}
    $\displaystyle3|B(e^{1}_{j},\eta^{1}_{k})|^{2}=\langle R(e^{1}_{j},\eta^{1}_{k})e^{1}_{j},\eta^{1}_{k}\rangle+ 2\lambda^{2}\langle e^{1}_{j},\eta^{1}_{k}\rangle^{2}+\lambda^{2}|e^{1}_{j}|^{2}|\eta^{1}_{k}|^{2}
    $
\end{center}
and 
\begin{center}
    $\displaystyle 3\langle B(\eta_{k}^{1}, \eta_{k}^{1}),B(e_{j}^{1},e_{j}^{1})\rangle=
    -2\langle R(e^{1}_{j},\eta^{1}_{k})e^{1}_{j},\eta^{1}_{k}\rangle  +\lambda^{2} |e^{1}_{j}|^{2}|\eta^{1}_{k}|^{2}+2\lambda^{2}\langle e^{1}_{j},\eta^{1}_{k}\rangle^{2} $.
\end{center}
Now applying the last two equalities in Equation (\ref{mainequation}), we obtain
\begin{align}
\sum_{A=1}^{m} -\langle N_{E_{A}},J_{\Sigma}(N_{E_{A}})\rangle &=     \sum_{j=1}^{n}\sum_{k=1}^{d}-\frac{4}{3} \langle R(e^{1}_{j},\eta^{1}_{k})\eta^{1}_{k},e^{1}_{j}\rangle+\frac{2\lambda^{2}}{3}\langle e^{1}_{j},\eta^{1}_{k}\rangle^{2}+\frac{\lambda^{2}}{3}|e^{1}_{j}|^{2}|\eta^{1}_{k}|^{2} \nonumber \\
\label{eqn: octo key eqn}
& = \sum_{j=1}^{n}\sum_{k=1}^{d}\frac{4}{3} \langle R(e^{1}_{j},\eta^{1}_{k})e^{1}_{j}, \eta^{1}_{k}\rangle+\frac{2\lambda^{2}}{3}\langle e^{1}_{j},\eta^{1}_{k}\rangle^{2}+\frac{\lambda^{2}}{3}|e^{1}_{j}|^{2}|\eta^{1}_{k}|^{2}.
\end{align}

We look at the summand. Notice that each term in the summand is invariant under isometries. Since $\Spin(9)$ is isomorphic to the isotropy group at $p$ and $\Spin(9)$ acts transitively and effectively on $S^{15}$, there exists some $g_j \in \Spin(9)$ such that $g_j(e_j^1) = (|e_j^1|,0) \in \O \oplus \O$. We then have $g_j(e_j^1) = (a, 0)$ for some $a \in \R$, $a>0$ and $g_j(\eta_k^1) = (c,d)$ for some $c,d \in \O$.

Also note that the setwise stabilizer of the 8-dimensional subspace $\O \oplus 0$ is isomorphic to $\Spin(8)$. For the vectors $(I_s, 0)$, $s = 0,\dots, 7$, let $h_s$ be an element of $\Spin(8) \subset \Spin(9)$ such that $h_s ((I_0, 0)) = (I_s, 0)$ and we take $h_0 = \Id$.

Now using Equation (\ref{eqn: curvature}) for the curvature tensor, we have 

\begin{equation*}
\begin{aligned}
& \frac{4}{3} \langle R(e^{1}_{j},\eta^{1}_{k})e^{1}_{j}, \eta^{1}_{k}\rangle+\frac{2\lambda^{2}}{3}\langle e^{1}_{j},\eta^{1}_{k}\rangle^{2}+\frac{\lambda^{2}}{3}|e^{1}_{j}|^{2}|\eta^{1}_{k}|^{2} \\
= & \frac{4}{3} \langle R\Big(g_j(e^{1}_{j}),g_j(\eta^{1}_{k})\Big)g_j(e^{1}_{j}), g_j(\eta^{1}_{k})\rangle+\frac{2\lambda^{2}}{3}\langle g_j(e^{1}_{j}),g_j(\eta^{1}_{k})\rangle^{2}+\frac{\lambda^{2}}{3}|g_j(e^{1}_{j})|^{2}|g_j(\eta^{1}_{k})|^{2} \\
= & \frac{4}{3} \langle R\Big((a,0),(c,d)\Big)(a,0), (c,d)\rangle + \frac{2\lambda^{2}}{3} \langle a, c \rangle ^2 + \frac{\lambda^2}{3} |a|^2\Big(|c|^2 +  |d|^2 \Big)\\
= & \frac{\lambda^2}{3}\bigg( -4|a|^2|c|^2 + 4\langle a,c\rangle^2  - |ad|^2\bigg) + \frac{2\lambda^{2}}{3} \langle a, c \rangle ^2 + \frac{\lambda^2}{3} \Big(|a|^2|c|^2 +  |a|^2|d|^2 \Big) \\
= &  {\lambda^2}\bigg(-|a|^2|c|^2 + 2\langle a,c\rangle^2   \bigg) \\
= &  {\lambda^2}\bigg(-|a|^2\sum_{s=0}^7 \langle I_s,c \rangle^2 + 2\langle a,c\rangle^2  \bigg) \\
= &  {\lambda^2}\bigg(-\sum_{s=0}^7 \big(\langle aI_s,c \rangle + \langle 0,d \rangle \big)^2 + 2\langle a,c\rangle^2   \bigg) \\
= &  {\lambda^2}\bigg(-\sum_{s=0}^7 \big(\langle h_sg_j(e_j^1), g_j(\eta_k^1) \rangle^2 + 2\langle g_j(e^{1}_{j}),g_j(\eta^{1}_{k})\rangle^2   \bigg)\\
= &  {\lambda^2}\bigg(-\sum_{s=0}^7 \big(\langle g_j^{-1}h_sg_j(e_j^1), \eta_k^1 \rangle^2 + 2\langle e^{1}_{j},\eta^{1}_{k}\rangle^2   \bigg).
\end{aligned}
\end{equation*}
Therefore Equation (\ref{eqn: octo key eqn}) becomes
\begin{equation}\label{eqn: octo computation}
\sum_{A=1}^{m} -\langle N_{E_{A}},J_{\Sigma}(N_{E_{A}})\rangle  =    {\lambda^2} \bigg(- \sum_{j=1}^n\sum_{k=1}^d\sum_{s=0}^7 \big(\langle g_j^{-1}h_sg_j(e_j^1), \eta_k^1 \rangle^2 + 2\sum_{j=1}^n\sum_{k=1}^d\langle e^{1}_{j},\eta^{1}_{k}\rangle^2   \bigg).
\end{equation}

Since $e_1,\dots, e_n, \eta_1,\dots. \eta_d$ form a basis of $T_{\Phi(p)}(\O P^2 \times M$), we have
$$
(e_j^1, 0) = \sum_{i=1}^n \langle (e_j^1, 0), e_i \rangle e_i + \sum_{k=1}^d \langle (e_j^1, 0), \eta_k \rangle \eta_k,$$
so looking at the projection onto $T_{\phi_1(p)}\O P^2$ yields
$$
e_j^1 = \sum_{i=1}^n \langle e_j^1, e_i^1 \rangle e_i^1 + \sum_{k=1}^d \langle e_j^1, \eta_k^1 \rangle \eta_k^1,$$
and taking inner product of $e_j^1$ with itself gives
\begin{equation}\label{eqn: octo norm substitution 1}
 |e^{1}_{j}|^{2} = \sum_{i=1}^n \langle e_j^1, e_i^1 \rangle^2 + \sum_{k=1}^d \langle e_j^1, \eta_k^1 \rangle^2.
\end{equation}
Similarly, using $g_j^{-1}h_sg_j(e_j^1)$ instead of $e_j^1$ gives 

\begin{equation}\label{eqn: octo norm substitution 2}
 |e^{1}_{j}|^{2} =|g_j^{-1}h_sg_j(e_j^1)|^{2} = \sum_{i=1}^n \langle g_j^{-1}h_sg_j(e_j^1), e_i^1 \rangle^2 + \sum_{k=1}^d \langle g_j^{-1}h_sg_j(e_j^1), \eta_k^1 \rangle^2.
\end{equation}

Substituting Equations (\ref{eqn: octo norm substitution 1}) and (\ref{eqn: octo norm substitution 2}) back into Equation (\ref{eqn: octo computation}), we have
\begin{align*}
 & \sum_{A=1}^{m} -\langle N_{E_{A}},J_{\Sigma}(N_{E_{A}})\rangle  \\
=&    {\lambda^2} \bigg(\sum_{j=1}^n\sum_{s=0}^7\Big(\sum_{i=1}^n \langle g_j^{-1}h_sg_j(e_j^1), e_i^1 \rangle^2 - |e_j^1|^2\Big) + 2\sum_{j=1}^n\Big(|e_j^1|^2-\sum_{i=1}^n\langle e^{1}_{j},e_i^1\rangle^2\Big)  \bigg) \\
=& {\lambda^2} \bigg(\sum_{j=1}^n\sum_{i=1}^n\sum_{s=0}^7 \langle h_sg_j(e_j^1), g_j(e_i^1) \rangle^2  - 2\sum_{j=1}^n\sum_{i=1}^n\langle e^{1}_{j},e_i^1\rangle^2 - 6\sum_{j=1}^n|e_j^1|^2  \bigg) \\
=& {\lambda^2} \bigg(\sum_{j=1}^n\sum_{i=1}^n\sum_{s=0}^7 \langle h_s((|e_j^1|,0)), g_j(e_i^1) \rangle^2  - 2\sum_{j=1}^n\sum_{i=1}^n\langle e^{1}_{j},e_i^1\rangle^2 - 6\sum_{j=1}^n\Big(\sum_{i=1}^n \langle e_j^1, e_i^1 \rangle^2 + \sum_{k=1}^d \langle e_j^1, \eta_k^1 \rangle^2\Big)  \bigg) \\
=& {\lambda^2} \bigg(\sum_{j=1}^n\sum_{i=1}^n\sum_{s=0}^7 |e_j^1|^2\langle (I_s,0), g_j(e_i^1) \rangle^2  - 8\sum_{j=1}^n\sum_{i=1}^n\langle e^{1}_{j},e_i^1\rangle^2 - 6\sum_{j=1}^n\sum_{k=1}^d \langle e_j^1, \eta_k^1 \rangle^2  \bigg) \\
=& {\lambda^2} \bigg(\sum_{j=1}^n\sum_{i=1}^n |e_j^1|^2|\Proj_{\O\oplus 0} g_j(e_i^1)|^2  - 8\sum_{j=1}^n\sum_{i=1}^n\langle e^{1}_{j},e_i^1\rangle^2 - 6\sum_{j=1}^n\sum_{k=1}^d \langle e_j^1, \eta_k^1 \rangle^2  \bigg) \\
=& {\lambda^2} \bigg(\sum_{j=1}^n\sum_{i=1}^n |e_j^1|^2|\Proj_{g_j^{-1}(\O\oplus 0)} e_i^1|^2  - 8\sum_{j=1}^n\sum_{i=1}^n\langle e^{1}_{j},e_i^1\rangle^2 - 6\sum_{j=1}^n\sum_{k=1}^d \langle e_j^1, \eta_k^1 \rangle^2  \bigg) \\
=& {\lambda^2} \bigg(\sum_{j=1}^n\sum_{i=1}^n |e_j^1|^2|\Proj_{L(e_j^1)} e_i^1|^2  - 8\sum_{j=1}^n\sum_{i=1}^n\langle e^{1}_{j},e_i^1\rangle^2 - 6\sum_{j=1}^n\sum_{k=1}^d \langle e_j^1, \eta_k^1 \rangle^2  \bigg),
\end{align*}
which is Equation (\ref{oMAIN1}).

Equation (\ref{oMAIN2}) follows from symmetry between $e_j$ and $\eta_k$ in Equation (\ref{eqn: octo key eqn}) and Equation (\ref{eqn: curvature}) for the curvature tensor.

\end{proof}

\begin{prop}\label{prop: octo ineq}
Let $x_1, \dots x_n$ be nonzero vectors  in $\O \oplus \O$. Then 

$$\sum_{j=1}^n\sum_{i=1}^n |x_j|^2|\Proj_{L(x_j)} x_i|^2  - 8\sum_{j=1}^n\sum_{i=1}^n\langle x_i,x_j\rangle^2 \le 0.$$
\end{prop}
\begin{proof}
Group the $x_i$'s according to the octonionic lines they belong to. Let $\ell_{m_1}, \dots, \ell_{m_k}$ be the resulting distinct octonionic lines. 
For each $\ell_{m_s}$, let $x_1^s, \dots, x_{n_s}^s$ be the vectors among $x_1, \dots, x_n$ that are on the line $\ell_{m_s}$.
Let $X_s$ be the $16 \times n_s$ matrix
$$X_s = [x_1^s \dots x_{n_s}^s].$$

Since $X_s^TX_s = [\langle x_i^s, x_j^s\rangle]_{i,j=1}^{n_s}$ is a Gram matrix, it is symmetric positive semi-definite. There thus exists some orthogonal matrix $Q_s$ of size $n_s$ such that $Q_s^TX_s^TX_sQ_s$ is diagonal. Further, since the columns of $X_s$ are from the same octonionic line, $X_s$ has at most rank $8$. Thus we can assume that $Q_s^TX_s^TX_sQ_s = \diag(\lambda_{s,1} \dots, \lambda_{s,8}, 0\dots, 0)$, where $\lambda_{s,1} \dots, \lambda_{s,8} \ge 0$ are the eigenvalues of $X_s^TX_s$.
The matrix identity 
$$[\langle (X_sQ_s)_i, (X_sQ_s)_j \rangle]_{i,j=1}^{n_s} = Q_s^TX_s^TX_sQ_s = \diag(\lambda_{s,1} \dots, \lambda_{s,8}, 0\dots, 0)$$
tells us that the $16 \times n_s$ matrix $X_sQ_s$ has orthogonal columns, and all but the first eight columns are zero vectors. Since $X_sQ_s$ and $X_s$ have the same column span, there exist orthonormal vectors $y_1^s,\dots, y_8^s \in \ell_{m_s}$ such that if we define an $16 \times 8$ matrix $Y_s$ and an $8\times n_s$ matrix $D_s$ as 
$$Y_s = [y_1^s \dots y_8^s] \quad \text{and} \quad D_s = [\diag(\sqrt{\lambda_{s,1}} \dots, \sqrt{\lambda_{s,8}}) \ 0 \dots 0],$$
then $X_sQ_s = Y_sD_s$. This gives $X_s = Y_sD_sQ_s^{T}$. 

The strategy for the proof is to rewrite 
\begin{center}
$\displaystyle\sum_{j=1}^n\sum_{i=1}^n |x_j|^2|\Proj_{L(x_j)} x_i|^2  - 8\sum_{j=1}^n\sum_{i=1}^n\langle x_i,x_j\rangle^2$
\end{center}
 in terms of the eigenvalues $\lambda_{s,j}$ so that we get a nicer expression, for which we can do maximization. In doing so, we need the $\ell_{m_s}$'s to be octonionic lines instead of aribitrary 8-dimensional spaces, and Lemma \ref{lem: octo line inner product} plays a crucial role.
 
We first compute
\begin{align*}
 \sum_{j=1}^n\sum_{i=1}^n\langle x_i,x_j\rangle^2 &= \sum_{r=1}^k\sum_{s=1}^k\sum_{i=1}^{n_r}\sum_{j=1}^{n_s}\langle x^r_i,x^s_j\rangle^2 \\
 &= \sum_{r=1}^k\sum_{s=1}^k \Tr((X_r^TX_s)^TX_r^TX_s) \\
 &=  \sum_{r=1}^k\sum_{s=1}^k\Tr(X_s^TX_rX_r^TX_s) \\
 &=  \sum_{r=1}^k\sum_{s=1}^k\Tr(X_rX_r^TX_sX_s^T) \\
 &=  \sum_{r=1}^k\sum_{s=1}^k\Tr(Y_rD_rQ_r^{T}(Y_rD_rQ_r^{T})^TY_sD_sQ_s^{T}(Y_sD_sQ_s^{T})^T) \\
  &=  \sum_{r=1}^k\sum_{s=1}^k\Tr\Big(Y_s^TY_rD_rD_r^{T}Y_r^TY_sD_sD_s^{T}).
\end{align*}

Lemma \ref{lem: octo line inner product} shows that $Y_r^TY_s = c_{r,s}A_{r,s}$ for some orthogonal matrix $A_{r,s}$ of size $8$, and some constant $0 \le c_{r,s} \le 1$. Also, $c_{r,s} = c_{s,r}$ and $A_{r,s} = A_{s,r}^T$. In particular, when $r =s$, we have $c_{r,r} = 1$ and $A_{r,r} = \Id_8$. 
Thus
\begin{align}
 \sum_{j=1}^n\sum_{i=1}^n\langle x_i,x_j\rangle^2 &= \sum_{r=1}^k\sum_{s=1}^{k}\Tr(Y_s^TY_rD_rD_r^{T}Y_r^TY_sD_sD_s^{T}) \nonumber \\
 & = \sum_{r=1}^k\sum_{s=1}^kc_{r,s}^2 \Tr(A_{r,s}^TD_rD_r^{T}A_{r,s}D_sD_s^{T}) \nonumber\\
& = \sum_{r=1}^k\sum_{s=1}^k\sum_{i=1}^8 \sum_{j=1}^8 c_{r,s}^2 \lambda_{r,i} \lambda_{s,j} (A_{r,s})_{ij}^2.
\label{eq: inner product}
\end{align}


We then look at the quantity
\begin{center}
    $\displaystyle\sum_{j=1}^n\sum_{i=1}^n |x_j|^2|\Proj_{L(x_j)} x_i|^2$.
\end{center}
Notice that since $X_s^TX_s = [\langle x_i^s, x_j^s \rangle ]_{i,j=1}^{n_s}$, we have

\begin{center}
    $\displaystyle \Tr(X_s^TX_s)=\sum_{j=1}^{n_s}|x_j^s|^2 $.
\end{center} 
On the other hand, since $X_s^TX_s$ is a symmetric matrix, its trace equals to the sum of its eigenvalues. Thus $\sum_{j=1}^{n_s}|x_j^s|^2 = \Tr(X_s^TX_s) = \sum_{j=1}^8 \lambda_{s,j}$. Using this and the identities $X_r = Y_rD_rQ_r^{T}$ and $Y_r^TY_s = c_{r,s}A_{r,s}$, we get 

\begin{align} 
\sum_{j=1}^n\sum_{i=1}^n |x_j|^2|\Proj_{L(x_j)} x_i|^2   & = \sum_{r=1}^k\sum_{s=1}^k\sum_{i=1}^{n_r}\sum_{j=1}^8 \lambda_{s,j}|\Proj_{\ell_{m_s}} x_i^r|^2 \nonumber \\
& = \sum_{r=1}^k\sum_{s=1}^k\sum_{j=1}^8 \lambda_{s,j} \sum_{i=1}^{n_r}\sum_{p=1}^{8}\langle x_i^r, y_p^s \rangle^2 \nonumber \\
& = \sum_{r=1}^k\sum_{s=1}^k\sum_{j=1}^8 \lambda_{s,j} \Tr((X_r^TY_s)^TX_r^TY_s) \nonumber \\
& = \sum_{r=1}^k\sum_{s=1}^k\sum_{j=1}^8 \lambda_{s,j} \Tr(Y_s^TX_rX_r^TY_s) \nonumber \\
& = \sum_{r=1}^k\sum_{s=1}^k\sum_{j=1}^8 \lambda_{s,j} \Tr(Y_sY_s^TX_rX_r^T) \nonumber \\
& = \sum_{r=1}^k\sum_{s=1}^k\sum_{j=1}^8 \lambda_{s,j} \Tr(Y_sY_s^T(Y_rD_rQ_r^{T})(Y_rD_rQ_r^{T})^T) \nonumber \\
& = \sum_{r=1}^k\sum_{s=1}^k\sum_{j=1}^8 \lambda_{s,j} \Tr(Y_sY_s^TY_rD_rQ_r^{T}Q_rD_r^TY_r^T) \nonumber \\
& = \sum_{r=1}^k\sum_{s=1}^k\sum_{j=1}^8 \lambda_{s,j} \Tr(Y_sY_s^TY_rD_rD_r^TY_r^T) \nonumber \\
& = \sum_{r=1}^k\sum_{s=1}^k\sum_{j=1}^8 \lambda_{s,j} \Tr(Y_s^TY_rD_rD_r^TY_r^TY_s) \nonumber \\
& = \sum_{r=1}^k\sum_{s=1}^k\sum_{j=1}^8 \lambda_{s,j} c_{r,s}^2\Tr(A_{r,s}^TD_rD_r^TA_{r,s}) \nonumber \\
& = \sum_{r=1}^k\sum_{s=1}^k\sum_{j=1}^8 \lambda_{s,j} c_{r,s}^2\Tr(A_{r,s}A_{r,s}^TD_rD_r^T) \nonumber \\
& = \sum_{r=1}^k\sum_{s=1}^k\sum_{j=1}^8 \lambda_{s,j} c_{r,s}^2\Tr(D_rD_r^T) \nonumber \\
& = \sum_{r=1}^k\sum_{s=1}^k\sum_{j=1}^8 \lambda_{s,j} c_{r,s}^2\sum_{i=1}^8 \lambda_{r,i}\nonumber \\
& = \sum_{r=1}^k\sum_{s=1}^k\sum_{i=1}^8\sum_{j=1}^8 c_{r,s}^2 \lambda_{r,i}\lambda_{s,j}. \label{eq: projection}
\end{align}

Using (\ref{eq: inner product}) and (\ref{eq: projection}), we have that 
\begin{align*}
 & \sum_{j=1}^n\sum_{i=1}^n |x_j|^2|\Proj_{L(x_j)} x_i|^2  - 8\sum_{j=1}^n\sum_{i=1}^n\langle x_i,x_j\rangle^2 \\
= & \sum_{r=1}^k\sum_{s=1}^k\sum_{i=1}^8\sum_{j=1}^8 c_{r,s}^2 \lambda_{r,i}\lambda_{s,j} - 8\sum_{r=1}^k\sum_{s=1}^k\sum_{i=1}^8 \sum_{j=1}^8 c_{r,s}^2 \lambda_{r,i} \lambda_{s,j} (A_{r,s})_{ij}^2 \\
=: & f(\dots, \lambda_{r,i}, \dots),
\end{align*}
where $f$ is a function of each $\lambda_{r,i}$.

Since 
$$\frac{\partial f}{\partial \lambda_{r,i}} = 2 \sum_{s=1}^k\sum_{j=1}^8 c_{r,s}^2\lambda_{s,j}  - 16\sum_{s=1}^k \sum_{j=1}^8 c_{r,s}^2 \lambda_{s,j} (A_{r,s})_{ij}^2,$$    
summing over $i$ yields 
\begin{equation}\label{eq: sum f}
\begin{aligned}
\sum_{i=1}^8 \frac{\partial f}{\partial \lambda_{r,i}} &= 16 \sum_{s=1}^k\sum_{j=1}^8 c_{r,s}^2\lambda_{s,j}  - 16\sum_{s=1}^k \sum_{j=1}^8 c_{r,s}^2 \lambda_{s,j}\sum_{i=1}^8 (A_{r,s})_{ij}^2 \\
&= 16 \sum_{s=1}^k\sum_{j=1}^8 c_{r,s}^2\lambda_{s,j}  - 16\sum_{s=1}^k \sum_{j=1}^8 c_{r,s}^2 \lambda_{s,j} \\
& = 0,
\end{aligned}
\end{equation}
where in the second to last step we used the fact that $A_{r,s}$ is an orthogonal matrix.

We maximize $f$ subject to the constraint

\begin{center}
    $\displaystyle \sum_{r=1}^k\sum_{i=1}^8 \lambda_{r,i}^2= C$
\end{center}
for some constant $C \ge 0$. Using Lagrange multipliers, at the maximum of $f$ subject to the constraint, there exists some $\alpha \in \R$ such that for all $r,i$,
\begin{align}\label{eq: Lagrange}
\frac{\partial f}{\partial \lambda_{r,i}} = \alpha \lambda_{r,i}.    
\end{align}
Using (\ref{eq: sum f}) and summing over $i$ yields 
\begin{align*}
0 = \sum_{i=1}^8\frac{\partial f}{\partial \lambda_{r,i}} & = \alpha \sum_{i=1}^8 \lambda_{r,i}.
\end{align*}
Thus either $\alpha = 0$, or $\sum_{i=1}^8 \lambda_{r,i} = 0$ for all $r$.

If $\alpha = 0$, then by Equation (\ref{eq: Lagrange}), we have $\frac{\partial f}{\partial \lambda_{r,i}} = 0$ for all $r,i$.
Thus 
\begin{align*}
0 & = \sum_{r=1}^k\sum_{i=1}^8 \lambda_{r,i} \frac{\partial f}{\partial \lambda_{r,i}} \\
& = 2 \sum_{r=1}^k\sum_{i=1}^8\sum_{s=1}^k\sum_{j=1}^8 c_{r,s}^2\lambda_{r,i}\lambda_{s,j}  - 16\sum_{r=1}^k\sum_{i=1}^8\sum_{s=1}^k \sum_{j=1}^8 c_{r,s}^2 \lambda_{r,i}\lambda_{s,j} (A_{r,s})_{ij}^2 \\
& = 2f.
\end{align*}
If instead $\sum_{i=1}^8 \lambda_{r,i} = 0$ for all $r$, then 
\begin{align*}
f & =\sum_{r=1}^k\sum_{s=1}^k\sum_{i=1}^8\sum_{j=1}^8 c_{r,s}^2 \lambda_{r,i}\lambda_{s,j} - 8\sum_{r=1}^k\sum_{s=1}^k\sum_{i=1}^8 \sum_{j=1}^8 c_{r,s}^2 \lambda_{r,i} \lambda_{s,j} (A_{r,s})_{ij}^2 \\   
& =\sum_{r=1}^k\sum_{s=1}^k\sum_{j=1}^8 c_{r,s}^2 \Big(\sum_{i=1}^8\lambda_{r,i} \Big)\Big(\sum_{j=1}^8\lambda_{s,j}\Big) - 8\sum_{r=1}^k\sum_{s=1}^k\sum_{i=1}^8 \sum_{j=1}^8 c_{r,s}^2 \lambda_{r,i} \lambda_{s,j} (A_{r,s})_{ij}^2 \\ 
 & = - 8\sum_{r=1}^k\sum_{s=1}^k\sum_{i=1}^8 \sum_{j=1}^8 c_{r,s}^2 \lambda_{r,i} \lambda_{s,j} (A_{r,s})_{ij}^2 \\ 
 & = -8  \sum_{j=1}^n\sum_{i=1}^n\langle x_i,x_j\rangle^2 \\
& \leq 0. 
\end{align*}
In either case, at the maximum of $f$ subject to $ \sum_{r=1}^k\sum_{i=1}^8 \lambda_{r,i}^2= C$, we have $f \le 0$. 

Since $C \ge 0$ is arbitrary, we have $f \le 0$ always, which gives the desired inequality.
\end{proof}

\begin{rem}
Notice that in the above proposition, if $n \le 8$, then the proof is significantly simpler: just use the trivial estimate $|\Proj_{L(x_j)} x_i| \le |x_i|$ and apply Proposition \ref{ineq: sum}.
\end{rem}

By Lemma \ref{GENERALO} and Proposition \ref{prop: octo ineq}, at each point $p \in \Sigma$, we have 
\begin{align*}
& \sum_{A=1}^{m} -\langle N_{E_{A}},J_{\Sigma}(N_{E_{A}})\rangle \\
=& {\lambda^2} \bigg(\sum_{j=1}^n\sum_{i=1}^n |e_j^1|^2|\Proj_{L(e_j^1)} e_i^1|^2  - 8\sum_{j=1}^n\sum_{i=1}^n\langle e^{1}_{j},e_i^1\rangle^2 - 6\sum_{j=1}^n\sum_{k=1}^d \langle e_j^1, \eta_k^1 \rangle^2  \bigg)  \\
\leq & - 6{\lambda^2}  \sum_{j=1}^n\sum_{k=1}^d \langle e_j^1, \eta_k^1 \rangle^2  \\
\leq & 0.
\end{align*}

Now recall $\Phi=(\phi_1,\phi_2):\Sigma\rightarrow \mathbb{O}P^{2}\times M$ is a stable compact minimal immersion. The stability of $\Phi$ requires
\begin{equation*}
    \int_\Sigma \sum_{A=1}^{m} -\langle N_{E_{A}},J_{\Sigma}(N_{E_{A}})\rangle \ge 0.
\end{equation*} 

Combining the above two inequalites, we must have
$$0 = \sum_{A=1}^{m} -\langle N_{E_{A}},J_{\Sigma}(N_{E_{A}})\rangle =  - 6{\lambda^2}  \sum_{j=1}^n\sum_{k=1}^d \langle e_j^1, \eta_k^1 \rangle^2.$$ 
Thus for all $j,k$,
$$\langle e_j^1, \eta_k^1 \rangle = 0.$$
But since we have $0 = \langle   e_j, \eta_k \rangle = \langle e_j^1, \eta_k ^1 \rangle + \langle  e_j^2,  \eta_k^2 \rangle$, this shows that we necessarily have
\begin{equation*}
    \langle  e_j^2,  \eta_k^2 \rangle =  0
\end{equation*}
as well.\\

For the product manifold $\O P^{2} \times M$, let 
$$P: T(\O P^{2}\times M)  \to T \O P^{2} \oplus 0 \subseteq T(\O P^{2}\times M),$$
$$Q: T(\O P^{2}\times M)\to  0 \oplus TM \subseteq T(\O P^{2}\times M)$$
denote the projection maps. Let $F = P-Q$, then $F^2 = I$. For $p \in \Sigma$, we claim that 
$F( d\Phi(T_{p}{\Sigma})) \subseteq d\Phi(T_{p}{\Sigma})$.
In fact, since $F(e_j) = F(e_j^1, e_j^2) =  (e_j^1, -e_j^2)$, for each $\eta_k$, we have
$$
 \langle  F(e_j), \eta_k\rangle   = \langle  e_j^1, \eta_k^1\rangle -  \langle  e_j^2, \eta_k^2\rangle = 0.
$$
Therefore, $F(e_j) \perp N_p \Sigma$. This shows $F(e_j) \subset d\Phi(T_p\Sigma)$, so $F( d\Phi(T_{p}{\Sigma})) \subseteq d\Phi(T_{p}{\Sigma})$. Thus $\Sigma$ is an invariant submanifold of $\O P^{2} \times M$.

Analogously, as in the quaternionic case, by Theorem 1 in \cite{xu2000submanifolds}, $\Sigma$ is isometric to a product manifold $\Sigma_1 \times \Sigma_2$, where $\Sigma_1$ is an immersion into $\O P^{2}$ and $\Sigma_2$ is an immersion into $M$. Since $\Phi:\Sigma \to \O P^{2} \times M$ is a stable compact minimal immersion, we further have that $\Sigma_1$ and $\Sigma_2$ are stable compact minimal immersions into $\O P^{2}$ and $M$, respectively. However, Ohnita \cite[Theorem~E]{ohnita1986stable} showed that the only nontrivial stable minimal immersed submanifolds of $\O P^2$ are precisely the octonion protective lines
$\O P^1 \cong S^8$. Using this result, we have thus established Theorem \ref{main3}.

\bibliographystyle{alpha}
\bibliography{main}

\begin{thebibliography}{OPPV13}

\bibitem[BG72]{brown1972riemannian}
Robert Brown and Alfred Gray.
\newblock Riemannian manifolds with holonomy group {S}pin(9).
\newblock {\em Differential Geometry (in honor of Kentaro Yano), Kinokunia,
  Tokyo}, pages 41--59, 1972.

\bibitem[CL21]{chodosh2021stable}
Otis Chodosh and Chao Li.
\newblock Stable minimal hypersurfaces in $\mathbf{R}^4$.
\newblock \url{https://arxiv.org/abs/2108.11462}, 2021.

\bibitem[CW13]{chen2013stable}
Hang Chen and Xianfeng Wang.
\newblock On stable compact minimal submanifolds of {R}iemannian product
  manifolds.
\newblock {\em Journal of Mathematical Analysis and Applications}, 2013.

\bibitem[dCP79]{do1979stable}
Manfredo do~Carmo and Chia-Kuei Peng.
\newblock Stable complete minimal surfaces in $\mathbb{R}^{3}$ are planes.
\newblock {\em Bulletin (New Series) of the American Mathematical Society},
  1(6):903--906, 1979.

\bibitem[FCS80]{fischer1980structure}
Doris Fischer-Colbrie and Richard Schoen.
\newblock The structure of complete stable minimal surfaces in 3-manifolds of
  non-negative scalar curvature.
\newblock {\em Communications on Pure and Applied Mathematics}, 33(2):199--211,
  1980.

\bibitem[HSV09]{held2009semi}
Rowena Held, Iva Stavrov, and Brian VanKoten.
\newblock ({S}emi-){R}iemannian geometry of (para-)octonionic projective
  planes.
\newblock {\em Differential Geometry and its Applications}, 27(4):464--481,
  2009.

\bibitem[Kot20]{kotrbaty2020integral}
Jan Kotrbat{\`y}.
\newblock {\em Integral geometry on the octonionic plane}.
\newblock PhD thesis, Friedrich-Schiller-Universit{\"a}t Jena, 2020.

\bibitem[LS73]{lawson1973stable}
Blaine Lawson and James Simons.
\newblock On stable currents and their application to global problems in real
  and complex geometry.
\newblock {\em Annals of Mathematics}, 1973.

\bibitem[Ohn86]{ohnita1986stable}
Yoshihiro Ohnita.
\newblock Stable minimal submanifolds in compact rank one symmetric spaces.
\newblock {\em Tohoku Mathematical Journal, Second Series}, 1986.

\bibitem[OPPV13]{ornea2013spin}
Liviu Ornea, Maurizio Parton, Paolo Piccinni, and Victor Vuletescu.
\newblock Spin(9) geometry of the octonionic {H}opf fibration.
\newblock {\em Transformation Groups}, 18(3):845--864, 2013.

\bibitem[Pog81]{pogorelov1981stability}
Aleksei~V. Pogorelov.
\newblock On the stability of minimal surfaces.
\newblock {\em Doklady Akademii Nauk}, 260(2):293--295, 1981.

\bibitem[RL22]{Ramirez3}
Alejandra Ramirez-Luna.
\newblock Stable submanifolds in the product of projective spaces.
\newblock \url{https://arxiv.org/abs/2202.00083}, 2022.

\bibitem[Sak77]{sakamoto1977planar}
Kunio Sakamoto.
\newblock Planar geodesic immersions.
\newblock {\em Tohoku Mathematical Journal, Second Series}, 29(1):25--56, 1977.

\bibitem[Sim68]{simons1968minimal}
James Simons.
\newblock Minimal varieties in {R}iemannian manifolds.
\newblock {\em Annals of Mathematics}, 1968.

\bibitem[TU14]{torralbo2014stable}
Francisco Torralbo and Francisco Urbano.
\newblock On stable compact minimal submanifolds.
\newblock {\em Proceedings of the American Mathematical Society}, 2014.

\bibitem[XN00]{xu2000submanifolds}
Senlin Xu and Yilong Ni.
\newblock Submanifolds of product {R}iemannian manifold.
\newblock {\em Acta Mathematica Scientia}, 20(2):213--218, 2000.

\end{thebibliography}

\end{document}